\documentclass[11pt,reqno]{amsart}

\usepackage{amssymb, amsmath, amsthm}
\usepackage{hyperref}
\usepackage[alphabetic,lite]{amsrefs}
\usepackage{verbatim}
\usepackage{amscd}   
\usepackage[all]{xy} 
\usepackage{youngtab} 
\usepackage{young} 
\usepackage{tikz}
\usepackage{ mathrsfs }
\usepackage{cases}

\newcommand{\defi}[1]{{\upshape\sffamily #1}}

\renewcommand{\a}{\alpha}
\renewcommand{\b}{\beta}

\renewcommand{\d}{\delta}
\newcommand{\bw}{\bigwedge}
\renewcommand{\det}{\textrm{det}}
\newcommand{\K}{\bb{K}}
\renewcommand{\ll}{\lambda}

\newcommand{\onto}{\twoheadrightarrow}
\newcommand{\oo}{\otimes}

\renewcommand{\t}{\underline{t}}
\newcommand{\x}{\underline{x}}
\newcommand{\y}{\underline{y}}
\newcommand{\z}{\underline{z}}

\newcommand{\rank}{\textrm{rank}}

\newcommand{\Ext}{\operatorname{Ext}}
\newcommand{\GL}{\operatorname{GL}}

\newcommand{\Spec}{\operatorname{Spec}}
\newcommand{\Sym}{\operatorname{Sym}}

\newcommand{\bb}[1]{\mathbb{#1}}
\newcommand{\ch}[1]{\chi_{#1}(z,w)}
\renewcommand{\rm}[1]{\textrm{#1}}
\newcommand{\mc}[1]{\mathcal{#1}}
\newcommand{\mf}[1]{\mathfrak{#1}}

\newcommand{\Qp}[2]{\mathcal{Q}_{#1}(#2)}
\newcommand{\tl}[1]{\tilde{#1}}
\newcommand{\ul}[1]{\underline{#1}}

\def\lra{\longrightarrow}

\newcommand*{\lhra}{\ensuremath{\lhook\joinrel\relbar\joinrel\rightarrow}}

\newtheorem{theorem}{Theorem}[section]
\newtheorem*{theorem*}{Theorem}
\newtheorem{lemma}[theorem]{Lemma}

\newtheorem{corollary}[theorem]{Corollary}
\newtheorem*{corollary*}{Corollary}

\newtheorem*{charext*}{Theorem on the Characters of Ext Modules}
\newtheorem*{ext*}{Theorem on the Growth of Ext Modules}
\newtheorem*{loccoh*}{Theorem on Local Cohomology with Support in Generic Determinantal Ideals}
\newtheorem*{regularity*}{Theorem on Regularity of Equivariant Ideals}

\theoremstyle{definition}

\newtheorem*{definition*}{Definition}

\theoremstyle{remark}
\newtheorem{remark}[theorem]{Remark}
\newtheorem*{remark*}{Remark}

\numberwithin{equation}{section}

\begin{document}

\title{Local cohomology with support in generic determinantal ideals}

\author{Claudiu Raicu}
\address{Department of Mathematics, Princeton University, Princeton, NJ 08544\newline
\indent Institute of Mathematics ``Simion Stoilow'' of the Romanian Academy}
\email{craicu@math.princeton.edu}

\author{Jerzy Weyman}
\address{Department of Mathematics, University of Connecticut, Storrs, CT 06269}
\email{jerzy.weyman@uconn.edu}

\subjclass[2010]{Primary 13D45, 14M12}

\date{\today}

\keywords{Local cohomology, determinantal ideals, regularity}

\begin{abstract} For positive integers $m\geq n\geq p$, we compute the $\GL_m\times\GL_n$--equivariant description of the local cohomology modules of the polynomial ring $S=\Sym(\bb{C}^m\oo\bb{C}^n)$ with support in the ideal of $p\times p$ minors of the generic $m\times n$ matrix. Our techniques allow us to explicitly compute all the modules $\Ext^{\bullet}_S(S/I_{\x},S)$, for $\x$ a partition and $I_{\x}$ the ideal generated by the irreducible sub-representation of $S$ indexed by $\x$. In particular we determine the regularity of the ideals $I_{\x}$, and we deduce that the only ones admitting a linear free resolution are the powers of the ideal of maximal minors of the generic matrix, as well as the products between such powers and the maximal ideal of $S$.
\end{abstract}

\dedicatory{To the memory of Andrei Zelevinsky}

\maketitle

\section{Introduction}\label{sec:intro}

Given positive integers $m\geq n$ and a field $\K$ of characteristic zero, we consider the space $\K^{m\times n}$ of $m\times n$ matrices, and the ring $S$ of polynomial functions on this space. For each $p=1,\cdots,n$ we define the ideal $I_p\subset S$ generated by the polynomial functions in $S$ that compute the $p\times p$ minors of the matrices in $\K^{m\times n}$. The goal of this paper is to describe for each $p$ the local cohomology modules $H^{\bullet}_{I_p}(S)$ of $S$ with support in the ideal $I_p$. The case $p=n$ was previously analyzed by the authors in joint work with Emily Witt \cite{RWW}. There is a natural action of the group $\GL_m\times\GL_n$ on $\K^{m\times n}$ and hence on $S$, and this action preserves each of the ideals $I_p$. This makes $H^{\bullet}_{I_p}(S)$ into $\GL_m\times\GL_n$--representations, and our results describe the characters of these representations explicitly. Our methods also allow us to determine explicitly the characters of all the modules $\Ext^{\bullet}_S(S/I,S)$, 
where $I$ is an ideal of $S$ generated by an irreducible $\GL_m\times\GL_n$--subrepresentation of $S$, and in particular determine the regularity of such ideals. It is an interesting problem to determine the minimal free resolutions of such ideals $I$, which is unfortunately only answered in a small number of cases. We hope that our results will help shed some light on this problem in the future.

We will adopt a basis-independent notation throughout the paper, writing $F$ (resp. $G$), for $\K$--vector spaces of dimension $m$ (resp. $n$), and thinking of $F^*\oo G^*$ as the space $\K^{m\times n}$ of $m\times n$ matrices, and of $S=\Sym(F\oo G)$ as the ring of polynomial functions on this space. $S$ is graded by degree, with the space of linear forms $F\oo G$ sitting in degree $1$. The \defi{Cauchy formula} \cite[Cor.~2.3.3]{weyman}
\begin{equation}\label{eq:cauchy}
S=\bigoplus_{\x=(x_1\geq\cdots\geq x_n\geq 0)}S_{\x}F\oo S_{\x}G 
\end{equation}
describes the decomposition of $S$ into a sum of irreducible $\GL(F)\times\GL(G)$--representations, indexed by partitions $\x$ with at most $n$ parts ($S_{\x}$ denotes the \defi{Schur functor} associated to $\x$). This decomposition respects the grading, the term corresponding to $\x$ being situated in degree $|\x|=x_1+\cdots+x_n$. We denote by $I_{\x}$ the ideal generated by $S_{\x}F\oo S_{\x}G$. If we write $(1^p)$ for the partition $\x$ with $x_1=\cdots=x_p=1$, $x_i=0$ for $i>0$, then $I_{(1^p)}$ is just another notation for the ideal $I_p$ of $p\times p$ minors. Our first result gives an explicit formula for the regularity of the ideals $I_{\x}$: 

\begin{regularity*}[Theorem~\ref{thm:regularity}]
 For a partition $\ul{x}$ with at most $n$ parts, letting $x_{n+1}=-1$ we have the following formula for the regularity of the ideal $I_{\x}$:
\[\rm{reg}(I_{\ul{x}})=\max_{\substack{p=1,\cdots,n \\ x_{p}>x_{p+1}}}(n\cdot x_p+(p-2)\cdot(n-p)).\]
In particular, the only ideals $I_{\x}$ which have a linear resolution are those for which $x_1=\cdots=x_n$ (i.e. powers $I_n^{x_1}$ of the ideal $I_n$ of maximal minors) or $x_1-1=x_2=\cdots=x_n$ (i.e. $I_n^{x_1-1}\cdot I_1$).
\end{regularity*}

The theorem above is a consequence of the explicit description of the modules $\Ext^{\bullet}_S(S/I_{\x},S)$ that we obtain in Theorem~\ref{thm:ExtIx}. This description is somewhat involved, so we avoid stating it for the moment. A key point is that the modules $\Ext^{\bullet}_S(S/I_{\x},S)$ \emph{grow} as we append new columns to the \emph{end} of the partition $\x$. More precisely, we can identify a partition $\x$ with its pictorial realization as a \defi{Young diagram} consisting of left-justified rows of boxes, with $x_i$ boxes in the $i$-th row: for example, $\x=(5,5,5,3)$ corresponds to
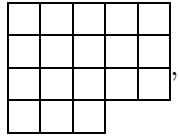
\[\Yvcentermath1\yng(5,5,5,3),\]
and adding two columns of size $2$ and three columns of size $1$ to the end of $\x$ yields $\y=(10,7,5,3)$.

\begin{ext*}[Theorem~\ref{thm:injectivityExt}]
 Let $d\geq 0$ and consider partitions $\x,\y$, where $\x$ consists of the first $d$ columns of $\y$, i.e. $x_i=\min(y_i,d)$ for all $i=1,\cdots,n$. The natural quotient map $S/I_{\y}\onto S/I_{\x}$ induces injective maps
\[\Ext^i_S(S/I_{\x},S)\lhra\Ext^i_S(S/I_{\y},S),\]
for all $i=0,1,\cdots,m\cdot n$. 
\end{ext*}

We warn the reader that the naive generalization of the statement above fails: if $\y$ is a partition containing $\x$ (i.e. $y_i\geq x_i$ for all $i$), then it is not always the case that the induced maps $\Ext^i_S(S/I_{\x},S)\to\Ext^i_S(S/I_{\y},S)$ are injective. In fact, a general partition $\x$ has the property that most modules $\Ext^i_S(S/I_{\x},S)$ are non-zero, but it is always contained in some partition $\y$ with $y_1=\cdots=y_n$; for such a $\y$, all but $n$ of the modules $\Ext^i_S(S/I_{\y},S)$ will vanish.

We next give the explicit description of $\Ext^{\bullet}_S(S/I_{\x},S)$, which requires some piece of notation. We write $\mf{R}$ for the representation ring of the group $\GL(F)\times\GL(G)$. Given a $\bb{Z}$--graded $S$--module $M$, admitting an action of $\GL(F)\times\GL(G)$ compatible with the natural one on $S$, we define its \defi{character} $\chi_M(z)$ to be the element in the Laurent power series ring $\mf{R}((z))$ given by
\[\chi_M(z)=\sum_{i\in\bb{Z}}[M_i]\cdot z^i,\]
where $[M_i]$ denotes the class in $\mf{R}$ of the $\GL(F)\times\GL(G)$--representation $M_i$. We will often work with doubly-graded modules $M_i^j$, where the second grading (in $j$) is a cohomological one, and $M_{\bullet}^j\neq 0$ only for finitely many values of $j$: for us they will be either $\Ext$ modules, or local cohomology modules. We define the character of such $M$ to be the element $\chi_M(z,w)\in\mf{R}((z))[w^{\pm 1}]$ given by
\[\chi_M(z,w)=\sum_{i,j\in\bb{Z}}[M_i^j]\cdot z^i\cdot w^j.\]

We will refer to an $r$--tuple $\ll=(\ll_1,\cdots,\ll_r)\in\bb{Z}^r$ (for $r=m$ or $n$) as a \defi{weight}. We say that $\ll$ is dominant if $\ll_1\geq\ll_2\geq\cdots\geq\ll_r$, and denote by $\bb{Z}^r_{dom}$ the set of dominant weights. Note that a partition is just a dominant weight with non-negative entries. We will usually use the notation $\x,\y,\z$ etc. to refer to partitions indexing the subrepresentations of $S$, and $\ll,\mu$ etc. to denote the weights describing the characters of other equivariant modules ($\Ext$ modules or local cohomology modules).

For $\ll\in\bb{Z}^n_{dom}$ and $0\leq s\leq n$, we define (note that in \cite{RWW} this was called $\ll(n-s)$)
\begin{equation}\label{eq:lls}
\ll(s)=(\ll_1,\cdots,\ll_s,\underbrace{s-n,\cdots,s-n}_{m-n},\ll_{s+1}+(m-n),\cdots,\ll_n+(m-n))\in\bb{Z}^m.
\end{equation}

\begin{charext*}[Theorem~\ref{thm:ExtIx}]
 With the above notation, the character of the doubly--graded module $\Ext^{\bullet}_S(S/I_{\x},S)$ is given by
\[
\ch{\Ext^{\bullet}_S(S/I_{\x},S)}=\sum_{\substack{1\leq p\leq n \\ 0\leq s\leq t_1\leq\cdots\leq t_{n-p}\leq p-1 \\ \ll\in W'(\x,p;\t,s)}}[S_{\ll(s)}F\oo S_{\ll}G]\cdot z^{|\ll|}\cdot w^{m\cdot n+1-p^2-s\cdot(m-n)-2\cdot\left(\sum_{j=1}^{n-p}t_j\right)}, 
\]
where $W'(\x,p;\t,s)$ is the set of dominant weights $\ll\in\bb{Z}^n$ satisfying
\[\begin{cases}
 \ll_n\geq p-x_p-m, & \\
 \ll_{t_j+j}\leq t_j-x_{n+1-j}-m, \rm{ for } j=1,\cdots,n-p, &\\
 \ll_s\geq s-n\textrm{ and }\ll_{s+1}\leq s-m.
 \end{cases}
\]
\end{charext*}

Our proof of this theorem starts with the observation in \cite{deconcini-eisenbud-procesi} that even though the algebraic set defined by $I_{\x}$ is somewhat simple (it is the set of matrices of rank smaller than the number of non-zero parts of $\x$), its scheme theoretic structure is more complicated: it is generally non-reduced, and has embedded components supported on $I_p$ for each size $p$ of some column of $\x$. Our approach is then to filter $S/I_{\x}$ with subquotients $J_{\z,p}$ (defined in Section~\ref{subsec:IxJxp}) whose scheme theoretic support is the (reduced) space of matrices of rank at most $p$, hence they are less singular and easier to resolve. In fact, each $J_{\z,p}$ is the push-forward of a locally free sheaf on some product of flag varieties, which allows us to compute $\Ext^{\bullet}_S(J_{\z,p},S)$ via duality theory. Solving the extension problem to deduce the formulas for $\Ext^{\bullet}_S(S/I_{\x},S)$ turns out to be then trivial, due to the restrictions imposed by the equivariant 
structure of the modules.

We end this introduction with our main theorem on local cohomology modules, whose statement needs some more notation. For $0\leq s\leq n$, we define (with the convention $\ll_0=\infty$, $\ll_{n+1}=-\infty$)
\begin{equation}\label{eq:defhs}
 h_s(z)=\sum_{\substack{\ll\in\bb{Z}^n_{dom} \\ \ll_s\geq s-n \\ \ll_{s+1}\leq s-m}} [S_{\ll(s)}F\oo S_{\ll}G]\cdot z^{|\ll|},
\end{equation}
so that $h_n(z)$ is just the character of $S$. The other $h_s(z)$'s are characters of local cohomology modules (in the case when $m>n$). More precisely, for $p=1,\cdots,n$ we write $H_p(z,w)$ for the character of the doubly-graded module $H^{\bullet}_{I_p}(S)$. In \cite{RWW} we proved that for $m>n$
\[H_n(z,w)=\sum_{s=0}^{n-1}h_s(z)\cdot w^{(n-s)\cdot(m-n)+1},\]
and it is easy to see that the same formula holds for $m=n$ (in this case, the only non-zero local cohomology module is $H^1_{I_n}(S)=S_{\rm{det}}/S$, where $\rm{det}$ denotes the determinant of the generic $n\times n$ matrix, and $S_{det}$ is the localization of $S$ at $\rm{det}$).

We write $p(a,b;c)$ for the number of partitions of $c$ contained in an $a\times b$ rectangle, and define the \defi{Gauss polynomial} ${a+b \choose b}(w)$ to be the generating function for the sequence $p(a,b;c)_{c\geq 0}$:
\begin{equation}\label{eq:gauss}
{a+b \choose a}(w)=\sum_{c\geq 0}p(a,b;c)\cdot w^c=\sum_{b\geq t_1\geq t_2\geq\cdots\geq t_a\geq 0}w^{t_1+\cdots+t_a}. 
\end{equation}
Gauss polynomials have previously appeared in \cite{akin-weyman} in connection to the closely related problem of understanding the minimal free resolutions of the ideals $I_{(p^d)}$.

\begin{loccoh*}[Theorem~\ref{thm:loccohdetl}]
With the above notation, we have for each $p=1,\cdots,n,$
  \[H_p(z,w) = \sum_{s=0}^{p-1} h_s(z)\cdot w^{(n-p+1)^2+(n-s)\cdot(m-n)}\cdot{n-s-1 \choose p-s-1}(w^2).\]
\end{loccoh*}

The theorem implies that the maximal cohomological index for which $H^{\bullet}_{I_p}(S)$ is non-zero (the \defi{cohomological dimension} of the ideal $I_p$) is obtained for $s=0$ and is equal to
\[(n-p+1)^2+n\cdot(m-n)+(p-1)\cdot(n-p)=m\cdot n-p^2+1.\]
This was first observed in \cite{bru-sch}. Our result says that \emph{many} of the local cohomology modules $H^j_{I_p}(S)$ are non-zero, for $j$ between the depth $(m-p+1)\cdot(n-p+1)$ of the ideal $I_p$ and its cohomological dimension. This contrasts with the positive characteristic situation where the only non-vanishing local cohomology module appears in degree $j=(m-p+1)\cdot(n-p+1)$ (see \cite[Cor.~4]{hochster-eagon} or \cite[Cor.~5.18]{bruns-vetter} where it is shown that $I_p$ is perfect, and \cite[Prop.~4.1]{peskine-szpiro} where a local cohomology vanishing result for perfect ideals in positive characteristic is proved). For determinantal ideals over arbitrary rings one can't expect such explicit results as Theorem~\ref{thm:loccohdetl}; for the latest advances in this general context, the reader should consult \cite{lyubeznik-singh-walther} and the references therein.

Our paper is organized as follows. In Section~\ref{sec:prelim} we give some representation-theoretic preliminaries: in Section~\ref{subsec:repthy} we fix some notation for Schur functors, weights and partitions; in Section~\ref{subsec:IxJxp} we recall from \cite{deconcini-eisenbud-procesi} some properties of the ideals $I_{\x}$, and introduce certain associated subquotients $J_{\x,p}$ that will play an important role in the sequel; in Section~\ref{subsec:bott} we recall the definition of flag varieties and formulate some consequences of Bott's theorem in a form that will be useful to us; we also recall in Section~\ref{subsec:ext} a method described in \cite{RWW} for computing extension groups for certain modules that arise as push-forwards of vector bundles with vanishing higher cohomology. In Section~\ref{sec:extmodulesJxp} we compute explicitly the characters of the modules $\Ext^{\bullet}_S(J_{\x,p},S)$, and in Section~\ref{sec:extmodulesIx} we use this calculation to deduce the main result about the 
characters of the modules $\Ext^{\bullet}_S(S/I_{\x},S)$ for all partitions $\x$. In Section~\ref{sec:regularity} we derive the formulas for the regularity of the ideals $I_{\x}$, while in Section~\ref{sec:determinantal} we describe the characters of the local cohomology modules with support in determinantal varieties.

\section{Preliminaries}\label{sec:prelim}

\subsection{Representation Theory {\cite{ful-har}, \cite[Ch.~2]{weyman}}}\label{subsec:repthy}
Throughout the paper, $\K$ will denote a field of characteristic $0$. If $W$ is a $\K$--vector space of dimension $\dim(W)=N$, a choice of basis determines an isomorphism between $\GL(W)$ and $\GL_N(\K)$. We will refer to $N$--tuples $\ll=(\ll_1,\cdots,\ll_N)\in\bb{Z}^N$ as \defi{weights} of the corresponding maximal torus of diagonal matrices. We say that $\ll$ is a \defi{dominant weight} if $\ll_1\geq\ll_2\geq\cdots\geq\ll_N$. Irreducible representations of $\GL(W)$ are in one-to-one correspondence with dominant weights $\ll$. We denote by $S_{\ll}W$ the irreducible representation associated to $\ll$, often referred to as a \defi{Schur functor}. We write $(a^N)$ for the weight with all parts equal to $a$, and define the \defi{determinant} of $W$ by $\det(W)=S_{(1^N)}W=\bw^N W$. We have $S_{\ll}W\oo\det(W)=S_{\ll+(1^N)}W$, and $S_{\ll}W^*=S_{(-\ll_N,\cdots,-\ll_1)}W$. We write $|\ll|$ for the total size $\ll_1+\cdots+\ll_N$ of $\ll$.
 
When $\x$ is a dominant weight with $x_N\geq 0$, we say that $\x$ is a \defi{partition} of $r=|\x|$. Note that when we're dealing with partitions we often omit the trailing zeros, so $\x=(5,2,1)$ is the same as $\x=(5,2,1,0,0,0)$. If $\y$ is another partition, we write $\x\subset\y$ to indicate that $x_i\leq y_i$ for all $i$.

\subsection{The ideals $I_{\x}$ and the subquotients $J_{\x,p}$}\label{subsec:IxJxp}
 Recall the Cauchy formula (\ref{eq:cauchy}) and the definition of the ideals $I_{\x}\subset S=\Sym(F\oo G)$ as the ideals generated by subrepresentations $S_{\x}F\oo S_{\x}G$ of $S$. It is shown in \cite{deconcini-eisenbud-procesi} that
\begin{equation}\label{eq:decompIx}
I_{\x}=\bigoplus_{\x\subset\y} S_{\y}F\oo S_{\y}G,
\end{equation}
and in particular $I_{\y}\subset I_{\x}$ if and only if $\x\subset\y$. More generally, for arbitrary partitions $\x,\y$, we let $\z=\max(\x,\y)$ be defined by $z_i=\max(x_i,y_i)$ for all $i$, and get
\begin{equation}\label{eq:intersectionIz}
 I_{\x}\cap I_{\y}=I_{\z}.
\end{equation}
Even more generally, for any set $T$ of partitions we let
\begin{equation}\label{eq:defIT}
I_T=\sum_{\y\in T}I_{\y}, 
\end{equation}
and have
\begin{equation}\label{eq:IxintersectIT}
I_{\x}\cap I_T=\sum_{\y\in T} I_{\max(\x,\y)}.
\end{equation}

For $p\in\{0,1,\cdots,n\}$ and $\x$ a partition, we write
\begin{equation}\label{eq:defSucc}
\rm{Succ}(\x,p)=\{\y:\x\subset\y,\rm{ and }y_i>x_i\rm{ for some }i>p\}. 
\end{equation}
By the discussion above, $I_{\y}\subset I_{\x}$ for all $\y\in\rm{Succ}(\x,p)$. We define
\begin{equation}\label{eq:defJxp}
J_{\x,p}=I_{\x}/I_{\rm{Succ}(\x,p)}
\end{equation}
It follows from (\ref{eq:decompIx}) that
\begin{equation}\label{eq:decompJxp}
J_{\x,p}=\bigoplus_{\substack{\x\subset\y \\ y_i=x_i\rm{ for all }i>p}} S_{\y}F\oo S_{\y}G.
\end{equation}
If $p=n$ then $J_{\x,p}=I_{\x}$, while if $p=0$ then $J_{\x,p}=S_{\x}F\oo S_{\x}G$ is just a vector space (it is annihilated by the maximal ideal of $S$). We have

\begin{lemma}\label{lem:succxp}
 Fix an index $p\in\{0,1,\cdots,n-1\}$, and consider a partition $\x$ with $x_1=\cdots=x_{p+1}$. Let
\begin{equation}\label{eq:defZ}
Z=\{\z:z_1=\cdots=z_{p+1}=x_1\}.
\end{equation}
We have
\begin{equation}\label{eq:succxinZ}
I_{\rm{Succ}(\x,p)}=\left(\sum_{\z\in Z,\ \x\subsetneq\z} I_{\z}\right) + I_{\max(\x,(x_1+1)^{p+1})}.
\end{equation}
\end{lemma}

\begin{proof}
 ``$\supset$'': Consider $\z\in Z$, $\x\subsetneq\z$. We have $z_i>x_i$ for some $i$, and since $x_i=z_i$ for $i\leq p+1$, we conclude that $z_i>x_i$ for some $i>p+1$, thus $\z\in\rm{Succ}(\x,p)$. Writing $\y=\max(\x,(x_1+1)^{p+1})$ we have that $y_{p+1}>x_{p+1}$ and $\y\supset\x$, so $\y\in\rm{Succ}(\x,p)$, proving that the RHS of (\ref{eq:succxinZ}) is contained in the LHS.
 
 ``$\subset$'': Consider a partition $\y\in\rm{Succ}(\x,p)$. If $y_{p+1}>x_{p+1}=x_1$ then $\y\supset\max(\x,(x_1+1)^{p+1})$, so $I_{\y}$ is contained in the RHS of (\ref{eq:succxinZ}). Otherwise $y_{p+1}=x_{p+1}$, so by possibly shrinking some of the first $p$ rows of $\y$ (which would enlarge $I_{\y}$), we may assume that $\y\in Z$. Clearly $\y\supsetneq\x$, since $y_i>x_i$ for some $i>p+1$, so it follows again that $I_{\y}$ is contained in the RHS of (\ref{eq:succxinZ}).
\end{proof}

The following result will be used in Section~\ref{sec:extmodulesIx}:

\begin{lemma}\label{lem:filtrationaddp+1}
 Fix an index $p\in\{0,1,\cdots,n-1\}$, and consider a partition $\x$ with $x_1=\cdots=x_{p+1}$. For a non-negative integer $d\geq 0$, let $\y$ be the partition defined by $y_i=x_i+d+1$ for $i=1,\cdots,p+1$, and $y_i=x_i$ for $i>p+1$ ($\y=\max(\x,(x_1+d+1)^{p+1})$). The quotient $I_{\x}/I_{\y}$ admits a filtration with successive quotients $J_{\z,p}$, where $\z$ runs over all partitions with 
\[
\begin{cases}
 x_1\leq z_1=\cdots=z_{p+1}\leq x_1+d, & \\
 z_i\geq x_i,\textrm{ for }i>p+1.
\end{cases}
\]
\end{lemma}

\begin{proof}
 By induction, it suffices to prove the result when $d=0$. We consider $Z$ as in (\ref{eq:defZ}) and define
\[\mc{I}(Z)=\{I_T:T\subset Z\}.\]
For $I\in\mc{I}(Z)$, we write
\[Z(I)=\{\z\in Z:I_{\z}\subset I\}.\]
Note that if $\z^0\in Z(I)$ then
\begin{equation}\label{eq:z0inz}
 \rm{if }\z\in Z\rm{ and }\z^0\subset\z\rm{ then }\z\in Z(I).
\end{equation}

We let $I_0=I_{((x_1+1)^{p+1})}$ and prove by induction on $|Z(I)|$ that for $I\in\mc{I}(Z)$, the quotient $(I+I_0)/I_0$ has a filtration with successive quotients $J_{\z,p}$, where $\z$ varies over the set of elements of $Z(I)$. Once we do this, we can take $I=I_{\x}$ and observe that $I_{\x}\cap I_0=I_{\y}$ (by (\ref{eq:intersectionIz})), which yields
\[(I+I_0)/I_0\simeq I/(I\cap I_0)=I_{\x}/I_{\y},\]
concluding the proof of the lemma.
 
For the induction, assume first that $|Z(I)|=1$, so that $I=I_{\z}$ with $z_1=\cdots=z_n=x_1$. We have $(I_{\z}+I_0)/I_0=J_{\z,p}$ so the base case for the induction follows.

Suppose now that $|Z(I)|>1$ and consider a \emph{maximal} element $\z^0$ in $Z(I)$, i.e. a partition $\z^0$ with the property that $I_{\z^0}\not\subset I_{\z}$ for any $\z\in Z(I)\setminus\{\z^0\}$. Define
\[I'=I_{Z(I)\setminus\{\z^0\}},\]
and note that $|Z(I')|=|Z(I)|-1$, $I=I'+I_{\z^0}$, and
\begin{equation}\label{eq:isoJxp}
(I+I_0)/(I'+I_0)\simeq J_{\z^0,p}, 
\end{equation}
which is proved as follows. The equality $I=I'+I_{\z^0}$ implies that the natural map
\[I_{\z^0}\to (I+I_0)/(I'+I_0)\]
is surjective. Its kernel is 
\[
\begin{split}
I_{\z^0}\cap(I'+I_0)\overset{(\ref{eq:IxintersectIT})}{=}&\left(\sum_{\z\in Z(I)\setminus\{\z^0\}} I_{\max(\z^0,\z)}\right) + I_{\max(\z^0,(x_1+1)^{p+1})}\\
\overset{(\ref{eq:z0inz})}{=}&\left(\sum_{\z\in Z,\ \z^0\subsetneq\z} I_{\z}\right) + I_{\max(\z^0,(x_1+1)^{p+1})}\overset{(\ref{eq:succxinZ})}{=}I_{\rm{Succ}(\z^0,p)},
\end{split}
\]
from which (\ref{eq:isoJxp}) follows. Since by induction $(I'+I_0)/I_0$ has a filtration with successive quotients $J_{\z,p}$ for $\z\in Z(I')$, we get the corresponding statement for $(I+I_0)/I_0$, finishing the induction step.
\end{proof}

\subsection{Partial flag varieties and Bott's Theorem {\cite[Ch.~4]{weyman}}}\label{subsec:bott}

Consider a $\K$--vector space $V$ with $\dim(V)=d$, and positive integers $q\leq n\leq d$. We denote by $Flag([q,n];V)$ the variety of partial flags
\[V_{\bullet}:\quad V\onto V_{n}\onto V_{n-1}\cdots\onto V_q\onto 0,
\]
where $V_p$ is a $p$--dimensional quotient of $V$ for each $p=q,q+1,\cdots,n$. For $p\in[q,n]$ we write $\Qp{p}{V}$ for the tautological rank $p$ quotient bundle on $Flag([q,n];V)$ whose fiber over a point $V_{\bullet}\in Flag([q,n];V)$ is $V_p$. For each $p$ there is a natural surjection of vector bundles
\begin{equation}\label{eq:flagsurjectionp}
V\oo\mc{O}_{Flag([q,n];V)}\onto\Qp{p}{V}. 
\end{equation}
Note that for $q=n$, $Flag([q,n];V)=\bb{G}(n,V)$ is the Grassmannian of $n$--dimensional quotients of $V$.

We consider the natural projection maps
\begin{equation}\label{eq:defpiq}
\pi^{(q)}:Flag([q,n];V)\to Flag([q+1,n];V), 
\end{equation}
defined by forgetting $V_q$ from the flag $V_{\bullet}$. For $q\leq n-1$, this map identifies $Flag([q,n];V)$ with the projective bundle $\bb{P}_{Flag([q+1,n];V)}(\Qp{p+1}{V})$, which comes with a tautological surjection
\[\Qp{p+1}{V}\onto\Qp{p}{V}.\]
For $q=n$ we make the convention $Flag([q+1,n];V)=\Spec(\K)$. With the usual notation $R^{\bullet}\pi^{(q)}_*$ for derived push-forward, we have the following consequence of Bott's theorem:

\begin{theorem}\label{thm:bott}
 (a) Suppose that $q\leq n-1$ and consider a dominant weight $\mu\in\bb{Z}^q$. For $q<p\leq n$
 \[R^j\pi^{(q)}_*(S_{\mu}\Qp{p}{V})=
 \begin{cases}
  S_{\mu}\Qp{p}{V} & \rm{if }j=0, \\
  0 & \rm{otherwise}.
 \end{cases}
\]
 If $\mu_{q-t}+t=-1$ for some $t=0,\cdots,q-1,$ then
 \[R^j\pi^{(q)}_*(S_{\mu}\Qp{q}{V})=0\rm{ for all }j.\]
 Otherwise (with the convention $\mu_0=\infty$, $\mu_{q+1}=-\infty$), there exists a unique index $0\leq t\leq q$ such that
\[\mu_{q-t+1}+t+1\leq 0\leq\mu_{q-t}+t.\]
Letting
\[\tl{\mu}=(\mu_1,\cdots,\mu_{q-t},-t,\mu_{q-t+1}+1,\cdots,\mu_q+1),\]
we have
\[R^j\pi^{(q)}_*(S_{\mu}\Qp{q}{V})=
\begin{cases}
S_{\tl{\mu}}\Qp{q+1}{V} & \rm{if }j=t,\\
0 & \rm{otherwise}.
\end{cases}
\]

 (b) Consider a dominant weight $\mu\in\bb{Z}^n$. If $n-d\leq \mu_{n-s}+s\leq -1$ for some $s=0,\cdots,n-1$ then
 \[R^j\pi^{(n)}_*(S_{\mu}\Qp{n}{V})=0\rm{ for all }j.\]
 Otherwise, (with the convention $\mu_0=\infty$, $\mu_{n+1}=-\infty$), there exists a unique index $0\leq s\leq n$ such that
\[\mu_{n-s}\geq -s\rm{ and }\mu_{n-s+1}\leq -s-d+n.\] 
Letting 
\[\tl{\mu}=(\mu_1,\cdots,\mu_{n-s},\underbrace{-s,\cdots,-s}_{d-n},\mu_{n-s+1}+(d-n),\cdots,\mu_n+(d-n))\in\bb{Z}^d,\]
(compare to (\ref{eq:lls})) we have
\[R^j\pi^{(n)}_*(S_{\mu}\Qp{n}{V})=
\begin{cases}
S_{\tl{\mu}}V & \rm{if }j=s\cdot(d-n), \\
0 & \rm{otherwise}.
\end{cases}
\]

\end{theorem}

\subsection{Computing Ext modules via duality}\label{subsec:ext} In this section we recall \cite[Thm.~3.1]{RWW} as a tool to compute $\Ext^{\bullet}_S(M,S)$ when $M$ comes as the push-forward of certain vector bundles with vanishing higher cohomology. More precisely, we have

\begin{theorem}\label{thm:duality}
 Let $X$ be a projective variety, and let $W$ be a finite dimensional $\K$--vector space. Suppose
\[W\oo\mc{O}_X\onto\eta,\]
is a surjective map, where $\eta$ is locally free, and let $k=\dim(W)-\rank(\eta)$. Consider a locally free sheaf $\mc{V}$ on $X$ and define
\[\mc{M}(\mc{V})=\mc{V}\oo\Sym(\eta),\quad\mc{M}^*(\mc{V})=\mc{V}\oo\det(W)\oo\det(\eta^*)\oo\Sym(\eta^*).\]
Giving $\mc{V}$ internal degree $v$, and $\eta$ and $W$ degree $1$, we think of $\mc{M}(\mc{V})$ and $\mc{M}^*(\mc{V})$ as graded sheaves, with
\[\mc{M}(\mc{V})_{i+v}=\mc{V}\oo\Sym^i(\eta),\ \mc{M}^*(\mc{V})_{i+v}=\mc{V}\oo\det(W)\oo\det(\eta^*)\oo\Sym^{-i+k}(\eta^*).\]
Suppose that $H^j(X,\mc{M}(\mc{V}))=0$ for $j>0$, and let
\[M(\mc{V})=H^0(X,\mc{M}(\mc{V})).\] 
We have for each $j\geq 0$ a graded isomorphism
\begin{equation}\label{eq:ExtH}
\Ext^j_S(M(\mc{V}),S)=H^{k-j}(X,\mc{M}^*(\mc{V}))^*, 
\end{equation}
where $(-)^*$ stands for the graded dual.
\end{theorem}

\section{Ext modules for the subquotients $J_{\x,p}$}\label{sec:extmodulesJxp}

The goal of this section is to compute explicitly the character of $\Ext^{\bullet}_S(J_{\x,p},S)$ for all $p$ and all partitions $\x$ with $x_1=\cdots=x_p$, where $J_{\x,p}$ is defined as in (\ref{eq:defJxp}). We will achieve this by realizing $J_{\x,p}$ as the global sections of a vector bundle with vanishing higher cohomology on a certain product of flag varieties, and then using the duality Theorem~\ref{thm:duality} and Bott's Theorem~\ref{thm:bott}.

Consider as before vector spaces $F,G,$ with $\dim(F)=m$, $\dim(G)=n$, $m\geq n$. For $q=1,\cdots,n,$ we consider the projective varieties
\[X^{(q)}=Flag([q,n];F)\times Flag([q,n];G),\ X=X^{(\infty)}=\Spec\K,\]
and the locally free sheaves (see Section~\ref{subsec:bott})
\[\eta^{(p)}=\Qp{p}{F}\oo\Qp{p}{G},\ p=1,\cdots,n,\ \eta=\eta^{(\infty)}=F\oo G.\]
Note that $\eta^{(p)}$ can be thought of as a sheaf on $X^{(q)}$ whenever $p\geq q$. We consider for $q\leq n-1$ (resp. $q=n$) the natural maps $\pi^{(q)}:X^{(q)}\to X^{(q+1)}$ (resp. $\pi^{(n)}:X^{(n)}\to X$) induced from (\ref{eq:defpiq}). We define 
\[S^{(q)}=\Sym\eta^{(q)}\]
as relative versions of the polynomial ring $S=S^{(\infty)}=\Sym(F\oo G)$. We will always work implicitly with quasi-coherent sheaves on the affine bundles
\[Y^{(q)}=\bb{A}_{X^{(q)}}(\eta^{(q)})=\ul{\Spec}_{X^{(q)}}(S^{(q)}),\]
which we identify with $S^{(q)}$--modules on $X^{(q)}$ as in \cite[Ex.~II.5.17]{hartshorne}. The Cauchy formula (\ref{eq:cauchy}) becomes in the relative setting
\begin{equation}\label{eq:cauchySq}
S^{(q)}=\bigoplus_{\x=(x_1\geq\cdots\geq x_q\geq 0)}S_{\x}\Qp{q}{F}\oo S_{\x}\Qp{q}{G}, 
\end{equation}
and we can define the ideals $I_{\x}^{(q)}\subset S^{(q)}$ and subquotients $J_{\x,p}^{(q)}$ for $0\leq p\leq q$, analogously to~(\ref{eq:decompIx}) and~(\ref{eq:defJxp}). For $1\leq p\leq q$, we write $I_p^{(q)}$ for $I_{(1^p)}^{(q)}$, the ideal of $p\times p$ minors in $S^{(q)}$. We define the line bundle
\begin{equation}\label{eq:detq}
\det^{(q)}=\det(\Qp{q}{F})\oo\det(\Qp{q}{G}),
\end{equation}
and note that the ideal $I_q^{(q)}$ is generated by $\det^{(q)}$. It follows easily from (\ref{eq:cauchySq}) and Theorem~\ref{thm:bott} that
\begin{equation}\label{eq:Rpi*Sp}
 R^j\pi^{(q)}_*(S^{(q)})=
\begin{cases}
 S^{(q+1)}/I^{(q+1)}_{q+1} & \rm{if }j=0, \\
 0 & \rm{otherwise},
\end{cases}
\rm{ and for }p>q,\ 
 R^j\pi^{(q)}_*(S^{(p)})=
\begin{cases}
 S^{(p)} & \rm{if }j=0, \\
 0 & \rm{otherwise}.
\end{cases}
\end{equation}

\begin{lemma}\label{lem:operationsIxJxp}
 (a) For a partition $\x=(x_1\geq\cdots\geq x_q)$, there exist natural identifications
\begin{equation}\label{eq:Ixq*det}
\det^{(q)}\oo I_{\x}^{(q)}=I_{\x+(1^q)}^{(q)},\rm{ and } 
\end{equation}
\begin{equation}\label{eq:Jxq*det}
\det^{(q)}\oo J_{\x,p}^{(q)}=J_{\x+(1^q),p}^{(q)},\rm{ for }0\leq p\leq q.
\end{equation}

(b) For a partition $\x=(x_1\geq\cdots\geq x_q)$ we have
\begin{equation}\label{eq:Rpi*Ixq}
R^j\pi^{(q)}_* I_{\x}^{(q)}=
\begin{cases}
 (I_{\x}^{(q+1)}+I_{q+1}^{(q+1)})/I_{q+1}^{(q+1)} & \rm{if }j=0, \\
 0 & \rm{otherwise}.
\end{cases}
\end{equation}

(c) For a partition $\x=(x_1\geq\cdots\geq x_q)$ and $0\leq p\leq q$ we have
\begin{equation}\label{eq:Rpi*Jxpq}
R^j\pi^{(q)}_* J_{\x,p}^{(q)}=
\begin{cases}
 J_{\x,p}^{(q+1)} & \rm{if }j=0, \\
 0 & \rm{otherwise}.
\end{cases}
\end{equation}
\end{lemma}

\begin{proof} (a) The multiplication map $\det^{(q)}\oo S^{(q)}\to S^{(q)}$ is injective: if we think of $S^{(q)}$ as locally the ring of polynomial functions on $q\times q$ matrices, then $\det^{(q)}$ is the determinant of the generic $q\times q$ matrix. It follows that $\det^{(q)}\oo I_{\x}^{(q)}=\det^{(q)}\cdot I_{\x}^{(q)}$ is in fact an ideal in $S^{(q)}$. (\ref{eq:Ixq*det}) then follows from the fact that multiplying by the determinant corresponds to adding a column of maximal size to the Young diagram (a special case of Pieri's rule). In fact, the same argument shows that for any set of partitions $Z$
\[\det^{(q)}\oo\left(\sum_{\z\in Z}I_{\z}^{(q)}\right)=\sum_{\z\in Z}I_{\z+(1^q)}^{(q)}.\]
Given the definition of $J_{\x,p}^{(q)}$ as the analogue of~(\ref{eq:defJxp}), (\ref{eq:Jxq*det}) follows by taking $Z=\rm{Succ}^{(q)}(\x,p)$ (the analogue of (\ref{eq:defSucc})) in the formula above, and using~(\ref{eq:Ixq*det}) and the exactness of tensoring with $\det^{(q)}$.

Part (b) follows from (\ref{eq:Rpi*Sp}), while (c) follows from the fact that if $\x=(x_1,\cdots,x_q)$ and $0\leq p\leq q$ then
\[\rm{Succ}^{(q+1)}(\x,p)=\rm{Succ}^{(q)}(\x,p)\bigcup\{\z:\z\supset\x,\ z_{p+1}\geq 1\}.\qedhere\]
\end{proof}

For each partition $\x=(x_1=\cdots=x_p\geq x_{p+1}\cdots\geq x_n\geq x_{n+1}=0)$, we define the locally free sheaf $\mc{M}_{\x,p}$ on $X^{(p)}$ by
\begin{equation}\label{eq:Mxp}
\mc{M}_{\x,p}=\left(\bigotimes_{q=p}^n(\det^{(q)})^{\oo(x_q-x_{q+1})}\right)\oo S^{(p)}.
\end{equation}

\begin{lemma}\label{lem:Rpi*Mxp}
 With the notation above, we have
\[H^j(X^{(p)},\mc{M}_{\x,p})=
\begin{cases}
 J_{\x,p} & \rm{if }j=0, \\
 0 & \rm{otherwise}.
\end{cases}
\]
\end{lemma}

\begin{proof}
 Note that $S^{(p)}=J_{\ul{0},p}^{(p)}$, so using (\ref{eq:Ixq*det}) we get
\[\mc{M}_{\x,p}=\left(\bigotimes_{q=p+1}^n(\det^{(q)})^{\oo(x_q-x_{q+1})}\right)\oo J_{((x_p-x_{p+1})^p),p}^{(p)}.\]
It follows that
\[
\begin{split}
R\pi^{(p)}_*\mc{M}_{\x,p}=\pi^{(p)}_*\mc{M}_{\x,p}\overset{(\ref{eq:Rpi*Jxpq})}{=}&\left(\bigotimes_{q=p+1}^n(\det^{(q)})^{\oo(x_q-x_{q+1})}\right)\oo J_{((x_p-x_{p+1})^p),p}^{(p+1)} \\
\overset{(\ref{eq:Ixq*det})}{=}&\left(\bigotimes_{q=p+2}^n(\det^{(q)})^{\oo(x_q-x_{q+1})}\right)\oo J_{((x_p-x_{p+2})^p,x_{p+1}-x_{p+2}),p}^{(p+1)}.
\end{split}
\]
Applying $R\pi^{(p+1)}_*, R\pi^{(p+2)}_*,\cdots,R\pi^{(n)}_*$ iteratively, and using (\ref{eq:Rpi*Jxpq}) and (\ref{eq:Ixq*det}) as above, we obtain
\[R\pi_*\mc{M}_{\x,p}=\pi_*\mc{M}_{\x,p}=J_{(x_p^p,x_{p+1},\cdots,x_n),p}\overset{(x_1=\cdots=x_p)}{=}J_{\x,p},\]
where $\pi=\pi^{(n)}\circ\cdots\circ\pi^{(p)}$ is the structure map $X^{(p)}\to\Spec\K$, concluding the proof of the lemma.
\end{proof}

We are now ready to prove the main result of this section:

\begin{theorem}\label{thm:Rpi*Mxp}
The character of the doubly--graded module $\Ext^{\bullet}_S(J_{\x,p},S)$ is given by
\begin{equation}\label{eq:characterExtJxp}
\ch{\Ext^{\bullet}_S(J_{\x,p},S)}=\sum_{\substack{0\leq s\leq t_1\leq\cdots\leq t_{n-p}\leq p \\ \ll\in W(\x,p;\t,s)}}[S_{\ll(s)}F\oo S_{\ll}G]\cdot z^{|\ll|}\cdot w^{m\cdot n-p^2-s\cdot(m-n)-2\cdot\left(\sum_{j=1}^{n-p}t_j\right)}, 
\end{equation}
where $W(\x,p;\t,s)$ is the set of dominant weights $\ll\in\bb{Z}^n$ with the properties
\begin{subnumcases}{}
 \ll_n\geq p-x_p-m, & \label{eq:restrllJxp:a}\\
 \ll_{t_j+j}=t_j-x_{n+1-j}-m, \rm{ for } j=1,\cdots,n-p, & \label{eq:restrllJxp:b}\\
 \ll_s\geq s-n\textrm{ and }\ll_{s+1}\leq s-m. \label{eq:restrllJxp:c}
\end{subnumcases}
\end{theorem}

\begin{remark}\label{rem:RWWThm43}
 If we take $p=n$ and $x_1=\cdots=x_n=d$ in the above theorem, we recover \cite[Thm.~4.3]{RWW}. The character of $J_{\x,n}=I_{\x}=I_n^d$ is
\[\ch{\Ext^{\bullet}_S(J_{\x,n},S)}=\sum_{\substack{0\leq s\leq n \\ \ll_n\geq n-d-m \\ \ll_s\geq s-n \\ \ll_{s+1}\leq s-m}}[S_{\ll(s)}F\oo S_{\ll}G]\cdot z^{|\ll|}\cdot w^{(n-s)\cdot(m-n)}.\] 
When $p=0$, since $J_{\x,0}=S_{\x}F\oo S_{\x}G$ is just a vector space, the only non-vanishing $\Ext$-module is
\[\Ext^{mn}_S(J_{\x,0},S)=\left(S_{\x}F\oo S_{\x}G\oo\det(F\oo G)\right)^*.\]
\end{remark}

\begin{proof}[Proof of Theorem~\ref{thm:Rpi*Mxp}]
 We apply Theorem~\ref{thm:duality} with 
\[X=X^{(p)},\ \eta=\eta^{(p)},\ W=F\oo G,\ \mc{V}=\bigotimes_{q=p}^n(\det^{(q)})^{\oo(x_q-x_{q+1})},\]
so that $\mc{M}(\mc{V})=\mc{M}_{\x,p}$ (see (\ref{eq:Mxp})). Lemma~\ref{lem:Rpi*Mxp} insures that the hypotheses of the duality theorem hold, and $M(\mc{V})=J_{\x,p}$. We have $\rank(\eta^{(p)})=p^2$, $\dim(W)=m\cdot n$, so $k=m\cdot n-p^2$. We give $\mc{V}$ internal degree $v=|\x|$ and get
\begin{equation}\label{eq:ExtJxp=H}
\Ext^j_S(J_{\x,p},S)_{r-|\x|}=H^{m\cdot n-p^2-j}\left(X^{(p)},\bigotimes_{q=p}^n(\det^{(q)})^{\oo(x_q-x_{q+1})}\oo\det(F\oo G)\oo\det(\eta^*)\oo\Sym^{r+m\cdot n-p^2}(\eta^*)\right)^*. 
\end{equation}
Formula~(\ref{eq:characterExtJxp}) now follows from a direct application of Theorem~\ref{thm:bott}, which we sketch below.

Using Cauchy's formula and the fact that $\det(\eta^*)=\det(\Qp{p}{F})^{-p}\oo\det(\Qp{p}{G})^{-p}$ we get
\[\det(\eta^*)\oo\Sym^{r+m\cdot n-p^2}(\eta^*)=\bigoplus_{\substack{\mu\in\bb{Z}^p_{dom} \\ |\mu|=r+m\cdot n \\ \mu_1\leq -p}}S_{\mu}\Qp{p}{F}\oo S_{\mu}\Qp{p}{G}.\]
For each $\mu$ in the formula above, we have to first compute
\begin{equation}\label{eq:Rpi*crazy}
R\pi_*\left(\bigotimes_{q=p}^n(\det^{(q)})^{\oo(x_q-x_{q+1})}\oo S_{\mu}\Qp{p}{F}\oo S_{\mu}\Qp{p}{G}\right), 
\end{equation}
where $\pi=\pi^{(n)}\circ\cdots\circ\pi^{(p)}:X^{(p)}\to\Spec\K$ is the structure map, then tensor with $\det(F\oo G)$ and dualize, in order to get the corresponding contribution to (\ref{eq:ExtJxp=H}). If (\ref{eq:Rpi*crazy}) is non-zero, then there exists uniquely determined dominant weights $\mu^{(q)},\d^{(q)}\in\bb{Z}^q$ for $q=p,\cdots,n$, and non-negative integers $t_{n-q}$, $q=p,\cdots,n-1$, and $s$, such that $\mu^{(p)}=\mu$, and if we write
\[\mc{M}^{(q)}=S_{\mu^{(q)}}\Qp{q}{F}\oo S_{\mu^{(q)}}\Qp{q}{G},\quad\mc{N}^{(q)}=S_{\d^{(q)}}\Qp{q}{F}\oo S_{\d^{(q)}}\Qp{q}{G},\rm{ for }q=p,\cdots,n,\]
then
\begin{equation}\label{eq:deltaq} 
\mc{N}^{(q)} = \mc{M}^{(q)} \oo (\det^{(q)})^{\oo(x_q-x_{q+1})},\rm{ for }q=p,\cdots,n,
\end{equation}
$2\cdot t_{n-q}$ is the unique integer $j$ for which $R^j\pi^{(q)}_*(\mc{N}^{(q)})\neq 0$, and
\begin{equation}\label{eq:muq+1}
 R^{2\cdot t_{n-q}}\pi^{(q)}_*(\mc{N}^{(q)})=\mc{M}^{(q+1)}\rm{ for }q=p,\cdots,n-1,
\end{equation}
and finally, $s\cdot(m-n)$ is the unique integer $j$ for which $R^j\pi^{(n)}_*(\mc{N}^{(n)})\neq 0$.

$\d^{(q)}$ is easy to determine, namely we get from (\ref{eq:deltaq}) that
\begin{equation}\label{eq:dqformula}
\d^{(q)}=\mu^{(q)}+((x_q-x_{q+1})^q). 
\end{equation}
Assuming we know $\d^{(q)}$, (\ref{eq:muq+1}) determines $t_{n-q}$ and $\mu^{(q)}$ according to Theorem~\ref{thm:bott}(a): $t_{n-q}$ is the unique integer $t$ with the property
\begin{equation}\label{eq:ineqsdqt}
\d^{(q)}_{q-t+1}+t+1\leq 0\leq \d^{(q)}_{q-t}+t
\end{equation}
and
\begin{equation}\label{eq:muq+1formula}
\mu^{(q+1)}=(\d^{(q)}_1,\cdots,\d^{(q)}_{q-t_{n-q}},-t_{n-q},\d^{(q)}_{q-t_{n-q}+1}+1,\cdots,\d^{(q)}_q+1). 
\end{equation}
It follows from (\ref{eq:muq+1formula}) and (\ref{eq:dqformula}) that
\[\d^{(q+1)}_{q+1-t_{n-q}}=-t_{n-q}+x_{q+1}-x_{q+2}\geq -t_{n-q},\]
so $t=t_{n-q}$ satisfies the RHS inequality in (\ref{eq:ineqsdqt}) with $q$ replaced by $(q+1)$, which forces $t_{n-(q+1)}\leq t_{n-q}$. It follows easily that
\begin{equation}\label{eq:d^i_q-t}
\d^{(i)}_{q+1-t_{n-q}}=-t_{n-q}+x_{q+1}-x_{i+1},\rm{ for }i=q+1,\cdots,n. 
\end{equation}
We have so far seen how to calculate $\mu^{(q)},\d^{(q)}$ for $q=p,\cdots,n,$ and $t_{n-q}$ for $q=p,\cdots,n-1$, so we're left with determining $s$. By Theorem~\ref{thm:bott}(b), $s$ is uniquely determined by the inequalities
\begin{equation}\label{eq:dnsineqs}
\d^{(n)}_{n-s}\geq -s\rm{ and }\d^{(n)}_{n-s+1}\leq -s-m+n, 
\end{equation}
and moreover
\begin{equation}\label{eq:Rpi*dual}
R^{s\cdot(m-n)}\pi^{(n)}_*(\mc{N}^{(n)})=S_{\tl{\d}}F\oo S_{\d}G, 
\end{equation}
where $\d=\d^{(n)}$ and
\[\tl{\d}=(\d_1,\cdots,\d_{n-s},(-s)^{m-n},\d_{n-s+1}+(m-n),\cdots,\d_n+(m-n)).\]
Since $\d^{(n)}_{n-t_1}=-t_1+x_n\geq -t_1$, it follows as before that $s\leq t_1$. Tensoring (\ref{eq:Rpi*dual}) with $\det(F\oo G)=\det(F)^{\oo n}\oo\det(G)^{\oo m}$ and dualizing, we obtain by putting everything together that
there exist integers $0\leq s\leq t_1\leq\cdots\leq t_{n-p}\leq p$ such that
\[R^{s\cdot(m-n)+2\cdot\sum_{j=1}^{n-p}t_j}\pi_*\left(\bigotimes_{q=p}^n(\det^{(q)})^{\oo(x_q-x_{q+1})}\oo\det(F\oo G)\oo S_{\mu}\Qp{p}{F}\oo S_{\mu}\Qp{p}{G}\right)^*=S_{\ll(s)}F\oo S_{\ll}G,\]
where $\ll(s)$ is defined as in (\ref{eq:lls}) and
\begin{equation}\label{eq:dtoll}
\ll_i=-m-\d_{n-i+1}\rm{ for all }i=1,\cdots,n. 
\end{equation}
We next check that $\ll\in W(\x,p;\t,s)$. Since $\mu_1\leq -p$, it follows from (\ref{eq:dqformula}) and (\ref{eq:muq+1formula}) that $\d_1\leq -p+x_p$, so $\ll_n=-\d_1-m\geq p-x_p-m$, i.e. (\ref{eq:restrllJxp:a}) holds. Letting $i=n$ in (\ref{eq:d^i_q-t}) we get $\d_{q+1-t_{n-q}}=-t_{n-q}+x_{q+1}$, so $\ll_{n-q+t_{n-q}}=t_{n-q}-x_{q+1}-m$, i.e. (\ref{eq:restrllJxp:b}) holds. Finally, (\ref{eq:restrllJxp:c}) follows from (\ref{eq:dnsineqs}).

We conclude from the discussion above that (\ref{eq:characterExtJxp}) holds, after possibly replacing $W(\x,p;\t,s)$ by a smaller set. To see that all weights $\ll\in W(\x,p;\t,s)$ in fact appear, one has to reverse the steps above in order to show that each $\ll$ can be reached from a certain weight $\mu$. We give the formula for $\mu$, and leave the details to the interested reader. We first define $\d\in\bb{Z}^n_{dom}$ by reversing (\ref{eq:dtoll}), $\d_i=-m-\ll_{n+1-i}$. Letting $t_{n-p+1}=p$ and $t_0=0$, we let for each $i=0,\cdots,n-p$, and $j=1,\cdots,t_{i+1}-t_i,$
\[\mu_{p-t_{i+1}+j}=\d_{n-i-t_{i+1}+j}+x_{n-i}+(n-p-i).\qedhere\]
\end{proof}

\begin{corollary}\label{cor:filtrationJxp}
 Fix an index $p\in\{0,1,\cdots,n\}$. Suppose that $M$ is an $S$--module with a compatible $\GL(F)\times\GL(G)$ action, admitting a finite filtration with successive quotients isomorphic to $J_{\x^j,p}$, for $j=1,\cdots,r$, where each $\x^j$ is a partition satisfying $x^j_1=x^j_2=\cdots=x^j_p$. We have a decomposition as $\GL(F)\times\GL(G)$--representations
\begin{equation}\label{eq:sumExtJxp}
\Ext^i_S(M,S)=\bigoplus_{j=1}^r \Ext^i_S(J_{\x^j,p},S),
\end{equation}
for each $i=0,1,\cdots,m\cdot n$. Equivalently, if
\begin{equation}\label{eq:ABC}
0\lra A\lra B\lra C\lra 0
\end{equation}
is a $\GL(F)\times\GL(G)$--equivariant short exact sequence of $S$--modules admitting filtrations as above, then for each $i=0,1,\cdots,m\cdot n,$ the induced sequences
\begin{equation}\label{eq:splitExtJxp}
0\lra\Ext^i_S(C,S)\lra\Ext^i_S(B,S)\lra\Ext^i_S(A,S)\lra 0
\end{equation}
are exact.
\end{corollary}

\begin{proof}
 Suppose that the conclusion of the Corollary fails, and consider modules $A,B,C$ sitting in an exact sequence (\ref{eq:ABC}) such that (\ref{eq:sumExtJxp}) holds for $A$ and $C$, but fails for $B$. In particular, not all sequences~(\ref{eq:splitExtJxp}) are exact, so there exists an index $i$ and a non-trivial connecting homomorphism
\[\Ext^i_S(A,S)\overset{\delta}{\lra}\Ext^{i+1}_S(C,S).\]
It follows that some irreducible representation of $\GL(F)\times\GL(G)$ appears in both $\Ext^i_S(A,S)$ and $\Ext^{i+1}_S(C,S)$. This is clearly impossible when $m=n$, because from (\ref{eq:sumExtJxp}) and (\ref{eq:characterExtJxp}) it follows that the cohomological degrees $j$ for which $\Ext^j_S(A,S)$ (resp. $\Ext^j_S(C,S)$) are non-zero satisfy 
\[j\equiv m\cdot n-p^2\ (\rm{mod }2).\]
When $m>n$, a similar argument applies: if $[S_{\ll(s)}F\oo S_{\ll}G]=[S_{\mu(t)}F\oo S_{\mu}G]$ (with $\ll,\mu,\ll(s)$ and $\mu(t)$ dominant) then it follows from (\ref{eq:lls}) that $\ll=\mu$ and $s=t$; moreover, we get from (\ref{eq:sumExtJxp}) and (\ref{eq:characterExtJxp}) that the cohomological degrees $j$ for which $S_{\ll(s)}F\oo S_{\ll}G$ appears in $\Ext^j_S(A,S)$ (resp. $\Ext^j_S(C,S)$) satisfy
\[j\equiv m\cdot n-p^2-s\cdot(m-n)\ (\rm{mod }2).\qedhere\]
\end{proof}

\section{Ext modules for $S/I_{\x}$}\label{sec:extmodulesIx}

In this section we will use the explicit calculation of $\Ext^{\bullet}_S(J_{\x,p},S)$ from the previous section in order to deduce a formula for the characters of $\Ext^{\bullet}_S(S/I_{\x},S)$ for all ideals $I_{\x}$. We begin with an important consequence of the results in the preceding section:

\begin{corollary}\label{cor:injectivityExtrectangles}
 Fix an index $p\in\{0,1,\cdots,n-1\}$, and positive integers $b>c>0$. If we let $\x=(c^{p+1})$, $\y=(b^{p+1})$, then the natural quotient map $S/I_{\y}\onto S/I_{\x}$ induces injective maps
\[\Ext^i_S(S/I_{\x},S)\lhra\Ext^i_S(S/I_{\y},S),\]
for all $i=0,1,\cdots,m\cdot n$. More generally, if $\z$ is any partition with $z_1=\cdots=z_{p+1}$ and $\x=\z+(c^{p+1})$, $\y=\z+(b^{p+1})$, then the quotient map $I_{\z}/I_{\y}\onto I_{\z}/I_{\x}$ induces injective maps
\[\Ext^i_S(I_{\z}/I_{\x},S)\lhra\Ext^i_S(I_{\z}/I_{\y},S),\]
for all $i=0,1,\cdots,m\cdot n$.
\end{corollary}

\begin{proof}
 Note that the former statement follows from the latter by taking $\z=\ul{0}$ and noting that $S=I_{\ul{0}}$. By Lemma~\ref{lem:filtrationaddp+1}, the modules $A=I_{\x}/I_{\y}$, $B=I_{\z}/I_{\y}$, and $C=I_{\z}/I_{\x}$ have finite filtrations with quotients isomorphic to $J_{\ul{t},p}$ for various partitions $\ul{t}$. Corollary~\ref{cor:filtrationJxp} applies to yield the desired conclusion.
\end{proof}

\begin{theorem}\label{thm:injectivityExt}
 Let $d\geq 0$ and consider partitions $\x,\y$, where $\x$ consists of the first $d$ columns of $y$, i.e. $x_i=\min(y_i,d)$ for all $i=1,\cdots,n$. The natural quotient map $S/I_{\y}\onto S/I_{\x}$ induces injective maps
\begin{equation}\label{eq:injectivityExt}
\Ext^i_S(S/I_{\x},S)\lhra\Ext^i_S(S/I_{\y},S), 
\end{equation}
for all $i=0,1,\cdots,m\cdot n$.
\end{theorem}

\begin{proof}
 Arguing inductively, it suffices to prove the result when all the columns of $\y$ outside $\x$ have the same size (say $p+1$, for $p\in\{0,1,\cdots,n-1\}$), i.e. $\y=\x+(a^{p+1})$ for some positive integer $a$. Note that this forces $x_1=x_2=\cdots=x_{p+1}=d$. We prove by descending induction on $p$ that the induced map (\ref{eq:injectivityExt}) is injective. When $p=n-1$, $\x=(d^n)$ and $\y=((d+a)^n)$, so the conclusion follows from Corollary~\ref{cor:injectivityExtrectangles} (or from the results in \cite{RWW}).
 
 Suppose now that $p<n-1$ and $\y=\x+(a^{p+1})$, $x_1=\cdots=x_{p+1}=d$. Let $\z$ be the partition consisting of the columns of $\x$ of size strictly larger than $p+1$, i.e. $z_i=\min(x_i,x_{p+2})$ for all $i=1,\cdots,n$. Consider the partitions $\tl{\x}$ (resp. $\tl{\y}$), defined by letting $\tl{x}_i=x_i$ (resp. $\tl{y}_i=y_i$) for $i\neq p+2$, and $\tl{x}_{p+2}=x_{p+1}$ (resp. $\tl{y}_{p+2}=y_{p+1}$). Alternatively, $\tl{\x}=\z+((d-x_{p+2})^{p+2})$, $\tl{\y}=\z+((d+a-x_{p+2})^{p+2})$. By induction, we obtain for all $i=0,1,\cdots,m\cdot n,$ inclusions
\begin{equation}\label{eq:auxinjectivity}
\Ext^i_S(S/I_{\z},S)\lhra\Ext^i_S(S/I_{\tl{\x}},S),\textrm{ and }\Ext^i_S(S/I_{\z},S)\lhra\Ext^i_S(S/I_{\tl{\y}},S).
\end{equation}
The natural commutative diagrams
\[
\xymatrix{
 S/I_{\tl{\x}} \ar[r]\ar[d] & S/I_{\z} \ar@{=}[d] \\
 S/I_{\x} \ar[r] & S/I_{\z} \\
}
\quad\quad\quad
\xymatrix{
 S/I_{\tl{\y}} \ar[r]\ar[d] & S/I_{\z} \ar@{=}[d] \\
 S/I_{\y} \ar[r] & S/I_{\z} \\
}
\]
induce commutative diagrams
\[
\xymatrix{
 {\Ext^i_S(S/I_{\z},S) \ar@{=}[d]\ar[r]} & {\Ext^i_S(S/I_{\tl{\x}},S)}  \\
 {\Ext^i_S(S/I_{\z},S) \ar[r]} & {\Ext^i_S(S/I_{\x},S) \ar[u]} \\
} 
\quad\quad\quad
\xymatrix{
 {\Ext^i_S(S/I_{\z},S) \ar@{=}[d]\ar[r]} & {\Ext^i_S(S/I_{\tl{\y}},S)}  \\
 {\Ext^i_S(S/I_{\z},S) \ar[r]} & {\Ext^i_S(S/I_{\y},S) \ar[u]} \\
} 
\]
Since the top maps are injective by (\ref{eq:auxinjectivity}), the bottom ones must be injective as well. We get for all $i=0,1,\cdots,m\cdot n,$ commutative diagrams
\[
\xymatrix{
0 \ar[r] & \Ext^i_S(S/I_{\z},S) \ar[r]\ar@{=}[d] & \Ext^i_S(S/I_{\x},S) \ar[r]\ar^{\a_i}[d] & \Ext^i_S(I_{\x}/I_{\z},S) \ar[r]\ar^{\b_i}[d] & 0 \\
0 \ar[r] & \Ext^i_S(S/I_{\z},S) \ar[r] & \Ext^i_S(S/I_{\y},S) \ar[r] & \Ext^i_S(I_{\y}/I_{\z},S) \ar[r] & 0 \\
}
\]
where the rows are exact. The maps $\b_i$ are injective by Corollary~\ref{cor:injectivityExtrectangles}, forcing the maps $\a_i$ to be injective as well. We conclude that the inclusion (\ref{eq:injectivityExt}) holds, finishing the proof of the theorem.
\end{proof}

\begin{theorem}\label{thm:ExtIx}
 The character of the doubly--graded module $\Ext^{\bullet}_S(S/I_{\x},S)$ is given by
\begin{equation}\label{eq:characterExtIx}
\ch{\Ext^{\bullet}_S(S/I_{\x},S)}=\sum_{\substack{1\leq p\leq n \\ 0\leq s\leq t_1\leq\cdots\leq t_{n-p}\leq p-1 \\ \ll\in W'(\x,p;\t,s)}}[S_{\ll(s)}F\oo S_{\ll}G]\cdot z^{|\ll|}\cdot w^{m\cdot n+1-p^2-s\cdot(m-n)-2\cdot\left(\sum_{j=1}^{n-p}t_j\right)}, 
\end{equation}
where $W'(\x,p;\t,s)$ is the set of dominant weights $\ll\in\bb{Z}^n$ satisfying
\begin{subnumcases}{}
 \ll_n\geq p-x_p-m, & \label{eq:restrllIx:a}\\
 \ll_{t_j+j}\leq t_j-x_{n+1-j}-m, \rm{ for } j=1,\cdots,n-p, & \label{eq:restrllIx:b}\\
 \ll_s\geq s-n\textrm{ and }\ll_{s+1}\leq s-m. \label{eq:restrllIx:c}
 \end{subnumcases}
\end{theorem}

\begin{remark}\label{rem:xp>xp+1}
 The condition $t_{n-p}\leq p-1$ in (\ref{eq:characterExtIx}) combined with the inequalities
\[t_{n-p}-x_{p+1}-m\geq\ll_{t_{n-p}+n-p}\geq\ll_n\geq p-x_p-m\]
obtained from (\ref{eq:restrllIx:b}) by letting $j=n-p$, shows that the only values of $p$ for which there may be a non-trivial contribution to (\ref{eq:characterExtIx}) are the ones for which $x_p>x_{p+1}$, or $p=n$. It follows that for $x_1=\cdots=x_n$, the only interesting value of $p$ is $p=n$, in which case $I_{\x}=J_{\x,n}$ and~(\ref{eq:characterExtIx}) follows from~(\ref{eq:characterExtJxp}) and the standard long exact sequence relating $\Ext^{\bullet}_S(S/I_{\x},S)$ to $\Ext^{\bullet}_S(I_{\x},S)$.
\end{remark}

\begin{proof}[Proof of Theorem~\ref{thm:ExtIx}]
 We do induction on the number of columns of $\x$. When $\x=\ul{0}$, $S/I_{\x}=0$, so $\Ext^{\bullet}_S(S/I_{\x},S)=0$. It follows that (\ref{eq:characterExtIx}) is verified in this case, since $W'(\ul{0},p;\t,s)$ is empty whenever $0\leq s\leq p-1$: to see this, note that
\[s-m\overset{(\ref{eq:restrllIx:c})}{\geq}\ll_{s+1}\geq\ll_n\overset{(\ref{eq:restrllIx:a})}{\geq} p-m\]
implies $s\geq p$, which is incompatible with the condition $s\leq p-1$.
 
 Suppose now that $\y$ is obtained from $\x$ by appending a column of size $(q+1)$, for some $q=0,\cdots,n-1$. This implies that $x_1=\cdots=x_{q+1}$, and $y_i=x_i+1$ for $1\leq i\leq q+1$. It follows from Theorem~\ref{thm:injectivityExt} that
\begin{equation}\label{eq:chIy-chIx}
\ch{\Ext^{\bullet}_S(S/I_{\y},S)}=\ch{\Ext^{\bullet}_S(S/I_{\x},S)}+\ch{\Ext^{\bullet}_S(I_{\x}/I_{\y},S)}, 
\end{equation}
and from Lemma~\ref{lem:filtrationaddp+1} that
\begin{equation}\label{eq:Ix/Iy}
\ch{\Ext^{\bullet}_S(I_{\x}/I_{\y},S)}=\sum_{\substack{\z=(z_1\geq\cdots\geq z_n\geq 0) \\ z_1=\cdots=z_{q+1}=x_1 \\ z_i\geq x_i,\ i>q+1}}\ch{\Ext^{\bullet}_S(J_{\z,q},S)}. 
\end{equation}
By Remark~\ref{rem:xp>xp+1}, since $x_1=\cdots=x_{q+1}$ and $y_1=\cdots=y_{q+1}$, the only relevant terms in (\ref{eq:characterExtIx}) (for both $\x$ and $\y$) are those for which $p\geq q+1$. For such $p$, $x_{n+1-j}=y_{n+1-j}$ whenever $1\leq j\leq n-p$, so condition~(\ref{eq:restrllIx:b}) is the same for $\x$ as it is for $\y$. (\ref{eq:restrllIx:c}) is clearly the same for both $\x$ and $\y$, and the same is true for (\ref{eq:restrllIx:a}) when $p\geq q+2$. We conclude that $W'(\x,p;\t,s)=W'(\y,p;\t,s)$ for $p\neq q+1$, from which it follows using (\ref{eq:chIy-chIx}) and the induction hypothesis that in order to prove (\ref{eq:characterExtIx}) for $\y$, it suffices to show that
\begin{equation}\label{eq:Ix/Iy:1}
\begin{aligned}
\ch{\Ext^{\bullet}_S(I_{\x}/I_{\y},S)} &= \sum_{\substack{p=q+1 \\ 0\leq s\leq t_1\leq\cdots\leq t_{n-p}\leq p-1 \\ \ll\in W'(\y,p;\t,s)\setminus W'(\x,p;\t,s)}}[S_{\ll(s)}F\oo S_{\ll}G]\cdot z^{|\ll|}\cdot w^{m\cdot n+1-p^2-s\cdot(m-n)-2\cdot\left(\sum_{j=1}^{n-p}t_j\right)} \\
&= \sum_{\substack{0\leq s\leq t_1\leq\cdots\leq t_{n-q-1}\leq q \\ \ll\in W'(\y,q+1;\t,s)\setminus W'(\x,q+1;\t,s)}}[S_{\ll(s)}F\oo S_{\ll}G]\cdot z^{|\ll|}\cdot w^{m\cdot n-q^2-2\cdot q-s\cdot(m-n)-2\cdot\left(\sum_{j=1}^{n-q-1}t_j\right)}
\end{aligned}
\end{equation}
Note that since (\ref{eq:restrllIx:b}) and (\ref{eq:restrllIx:c}) are the same for $\x$ and $\y$ when $p=q+1$, it follows that 
\begin{equation}\label{eq:llinW'y-W'x}
 \ll\in W'(\y,q+1;\t,s)\setminus W'(\x,q+1;\t,s)\Leftrightarrow
\begin{cases}
 \ll_n=q+1-y_{q+1}-m=q-x_1-m, & \\
 \ll_{t_j+j}\leq t_j-x_{n+1-j}-m, \rm{ for } j=1,\cdots,n-q-1, & \\
 \ll_s\geq s-n\textrm{ and }\ll_{s+1}\leq s-m.
\end{cases}
\end{equation}

Consider now a partition $\z$ appearing in (\ref{eq:Ix/Iy}). We claim that $W(\z,q;\t,s)$ is empty unless $t_{n-q}=q$. Furthermore, if $\ll\in W(\z,q;\t,s)$, then $\ll_n=q-x_{1}-m$. To see this, note that
\[\ll_n\leq\ll_{t_{n-q}+n-q}\overset{(\ref{eq:restrllJxp:b})}{=}t_{n-q}-z_{q+1}-m=t_{n-q}-z_q-m\leq q-z_q-m\overset{(\ref{eq:restrllJxp:a})}{\leq}\ll_n,\]
which forces equalities throughout, and in particular 
\[t_{n-q}=q\rm{ and }\ll_n=t_{n-q}-z_{q+1}-m=q-x_1-m.\]
We get from (\ref{eq:characterExtJxp}) that
\begin{equation}\label{eq:charJzp}
\ch{\Ext^{\bullet}_S(J_{\z,q},S)}=\sum_{\substack{0\leq s\leq t_1\leq\cdots\leq t_{n-q-1}\leq t_{n-q}=q \\ \ll\in W(\x,q;\t,s)}}[S_{\ll(s)}F\oo S_{\ll}G]\cdot z^{|\ll|}\cdot w^{m\cdot n-q^2-2q-s\cdot(m-n)-2\cdot\left(\sum_{j=1}^{n-q-1}t_j\right)}.
\end{equation}
Combining (\ref{eq:Ix/Iy}), (\ref{eq:Ix/Iy:1}) and (\ref{eq:charJzp}), it remains to show that
\[W'(\y,q+1;\t,s)\setminus W'(\x,q+1;\t,s)=\bigcup_{\substack{\z=(z_1\geq\cdots\geq z_n\geq 0) \\ z_1=\cdots=z_{q+1}=x_1 \\ z_i\geq x_i,\ i>q+1}}W(\x,q;\t,s).\]
This follows immediately from (\ref{eq:llinW'y-W'x}) and the fact that the condition
\[\ll_{t_j+j}\leq t_j-x_{n+1-j}-m, \rm{ for } j=1,\cdots,n-q-1\]
in (\ref{eq:llinW'y-W'x}) is equivalent to the existence of a partition $\z$ satisfying $z_1=\cdots=z_{q+1}=x_1$, and $z_{n+1-j}\geq x_{n+1-j}$ for $j=1,\cdots,n-q-1$ (or equivalently $z_i\geq x_i$ for $i>q+1$), such that
\[\ll_{t_j+j}= t_j-z_{n+1-j}-m, \rm{ for } j=1,\cdots,n-q-1.\qedhere\]
\end{proof}

\section{Regularity of the ideals $I_{\x}$}\label{sec:regularity}

The explicit description of the character of $\Ext^{\bullet}_S(I_{\x},S)$ obtained in Theorem~\ref{thm:ExtIx} allows us to obtain the following result on the regularity of every ideal $I_{\x}$, whose proof will be the focus of the current section.

\begin{theorem}\label{thm:regularity}
For a partition $\ul{x}$ with at most $n$ parts, letting $x_{n+1}=-1$ we have the following formula for the regularity of the ideal $I_{\x}$:
\begin{equation}\label{eq:regIx}
\rm{reg}(I_{\ul{x}})=\max_{\substack{p=1,\cdots,n \\ x_{p}>x_{p+1}}}(n\cdot x_p+(p-2)\cdot(n-p)).
\end{equation}
In particular, the only ideals $I_{\x}$ which have a linear resolution are those for which $x_1=\cdots=x_n$ (i.e. powers $I_n^{x_1}$ of the ideal $I_n$ of maximal minors) or $x_1-1=x_2=\cdots=x_n$ (i.e. $I_n^{x_1-1}\cdot I_1$).
\end{theorem}

By \cite[Prop.~20.16]{eisCA}, one can compute the regularity of a finitely generated $S$--module $M$ by the formula
\begin{equation}\label{eq:regM}
\rm{reg}(M)=\max\{-r-j:\Ext^j_S(M,S)_r\neq 0\}. 
\end{equation}
Since $\rm{reg}(I_{\x})=\rm{reg}(S/I_{\x})+1$, we get by combining (\ref{eq:regM}) and (\ref{eq:characterExtIx}) that
\begin{equation}\label{eq:regIx1}
\rm{reg}(I_{\x})=\max_{\substack{1\leq p\leq n \\ 0\leq s\leq t_1\leq\cdots\leq t_{n-p}\leq p-1 \\ \ll\in W'(\x,p;\t,s)}}\left(-|\ll|-mn+p^2+s\cdot(m-n)+2\cdot\left(\sum_{j=1}^{n-p}t_j\right)\right). 
\end{equation}

It is then important to decide when $W'(\x,p;\t,s)$ is non-empty, which we do in the following lemma.

\begin{lemma}\label{lem:nonemptyW'}
 Fix $p\in\{1,\cdots,n\}$ and $0\leq s\leq t_1\leq\cdots\leq t_{n-p}\leq p-1$. The set $W'(\x,p;\t,s)$ is non-empty if and only if
\begin{equation}\label{eq:nonemptyconds}
\begin{cases}
 x_p-x_{n+1-j} \geq p-t_j,\rm{ for }j=1,\cdots,(n-p), & \\
 s\geq p-x_p.
\end{cases}
\end{equation}
Moreover, the weight $\ll\in W'(\x,p;\t,s)$ of minimal size (i.e. for which the quantity $-|\ll|$ is maximal) is given by
\begin{equation}\label{eq:minimizingll}
\begin{cases}
 \ll_1=\cdots=\ll_s=(s-n), & \\
 \ll_{s+1}=\cdots=\ll_n=(p-x_p-m).
\end{cases}
\end{equation}
\end{lemma}

\begin{proof}
 If $W'(\x,p;\t,s)$ is non-empty, then for any $\ll\in W'(\x,p;\t,s)$ we have
\[t_j-x_{n+1-j}-m\overset{(\ref{eq:restrllIx:b})}{\geq}\ll_{t_j+j}\geq\ll_n\overset{(\ref{eq:restrllIx:a})}{\geq} p-x_p-m,\]
for $j=1,\cdots,(n-p)$, and
\[s-m\overset{(\ref{eq:restrllIx:c})}{\geq}\ll_{s+1}\geq\ll_n\overset{(\ref{eq:restrllIx:a})}{\geq} p-x_p-m,\]
so (\ref{eq:nonemptyconds}) holds.

Conversely, assume that (\ref{eq:nonemptyconds}) holds, and define $\ll$ via (\ref{eq:minimizingll}). It is immediate to check that $\ll$ satisfies~(\ref{eq:restrllIx:a}--\ref{eq:restrllIx:c}), so $\ll\in W'(\x,p;\t,s)$ and the set is non-empty. The fact that this $\ll$ has minimal size follows from the fact that any other $\ll\in W'(\x,p;\t,s)$ is dominant and thus satisfies $\ll_1\geq\cdots\geq\ll_s\geq(s-n)$, and $\ll_{s+1}\geq\cdots\geq\ll_n\geq(p-x_p-m)$, so
\[|\ll|\geq s\cdot(s-n)+(n-s)\cdot(p-x_p-m)=(n-s)\cdot(p-x_p-s-m).\qedhere\]
\end{proof}

Lemma~\ref{lem:nonemptyW'} allows us to rewrite (\ref{eq:regIx1}) in the form
\begin{equation}\label{eq:regIx2}
\begin{aligned}
\rm{reg}(I_{\x}) & =\max_{\substack{1\leq p\leq n \\ 0\leq s\leq t_1\leq\cdots\leq t_{n-p}\leq p-1 \\ x_p-x_{n+1-j} \geq p-t_j \\ s\geq p-x_p}}\left(-(n-s)\cdot(p-x_p-s-m)-mn+p^2+s\cdot(m-n)+2\cdot\left(\sum_{j=1}^{n-p}t_j\right)\right) \\
& = \max_{\substack{1\leq p\leq n \\ 0\leq s\leq t_1\leq\cdots\leq t_{n-p}\leq p-1 \\ x_p-x_{n+1-j} \geq p-t_j \\ s\geq p-x_p}}\left(s\cdot(p-x_p-s)+n\cdot(x_p-p)+p^2+2\cdot\left(\sum_{j=1}^{n-p}t_j\right)\right) \\
& = \max_{\substack{1\leq p\leq n \\ 0\leq s\leq p-1 \\ x_p-x_{p+1} \geq 1 \\ s\geq p-x_p}}\left(s\cdot(p-x_p-s)+n\cdot x_p+(p-2)\cdot(n-p).\right)
\end{aligned}
\end{equation}

Since $s\geq p-x_p$, we have $s\cdot(p-x_p-s)\leq 0$, with equality if $s=0$ or $s=p-x_p$. For $1\leq p\leq n-1$, the condition $x_p-x_{p+1}\geq 1$ forces $x_p\geq 1$, so $p-x_p\leq p-1$. It follows that we can take $s=\max(0,p-x_p)$ in order to maximize the expression above. Likewise, if $p=n$ and $x_n\geq 1$, we take $s=\max(0,n-x_n)$. It follows that when $x_n\geq 1$, (\ref{eq:regIx2}) reduces to (\ref{eq:regIx}). However, if $x_n=0$ then for $p=n$ the conditions $s\leq p-1$ and $s\geq p-x_p$ are incompatible, so (\ref{eq:regIx2}) reduces to
\[\rm{reg}(I_{\x})=\max_{\substack{p=1,\cdots,n-1 \\ x_{p}>x_{p+1}}}(n\cdot(x_p-p)+p^2+2\cdot(p-1)\cdot(n-p)).\]
To see that this is the same as (\ref{eq:regIx}) it suffices to observe that $\rm{reg}(I_{\x})\geq nx_n=0$ (which is the term corresponding to $p=n$).

To finish the proof of the theorem, we need to verify the last assertion. Note that $I_{\x}$ is generated in degree $x_1+\cdots+x_n$, so it has a linear resolution if and only if 
\begin{equation}\label{eq:reg=gendeg}
\rm{reg}(I_{\x})=x_1+\cdots+x_n. 
\end{equation}
When $x_1=\cdots=x_n$, (\ref{eq:regIx}) reduces to the term with $p=n$, whose value is $nx_n=x_1+\cdots+x_n$. For $x_1-1=x_2=\cdots=x_n$, the only surviving terms in (\ref{eq:regIx}) are those with $p=1$ and $p=n$, so we get
\[\rm{reg}(I_{\x})=\max(n\cdot(x_1-1)+1,nx_n)=n\cdot(x_1-1)+1=x_1+\cdots+x_n.\]
Conversely, assume that (\ref{eq:reg=gendeg}) holds, and that the $x_i$'s aren't all equal. Take $p$ minimal with the property that $x_p>x_{p+1}$, so that $p<n$, $x_1=\cdots=x_p$ and $x_i\leq x_p-1$ for $i>p$. We have
\[\rm{reg}(I_{\x})\geq n\cdot(x_p-p)+p^2+2\cdot(p-1)\cdot(n-p)=px_p+(n-p)\cdot(x_p-1)+(n-p)\cdot(p-1)\geq x_1+\cdots+x_n,\]
with equality when $x_i=x_p-1$ for $i>p$ and $(n-p)\cdot(p-1)=0$. This forces $p=1$ and $x_1-1=x_2=\cdots=x_n$, concluding the proof of the theorem.

\section{Local cohomology with support in determinantal ideals}\label{sec:determinantal}

In this section we prove our main theorem on local cohomology with support in determinantal ideals. Recall that $S=\Sym(\bb{C}^m\oo\bb{C}^n)$ denotes the polynomial ring of functions on the space of $m\times n$ matrices, $I_p\subset S$ is the ideal of $p\times p$ minors of the generic $m\times n$ matrix, and $H_p(z,w)$ is the character of the doubly graded module $H^{\bullet}_{I_p}(S)$. Recall also the definition (\ref{eq:defhs}) of $h_s(z)$, and the notation (\ref{eq:gauss}) for Gauss polynomials.

\begin{theorem}\label{thm:loccohdetl}
 We have for each $p=1,\cdots,n,$
  \[H_p(z,w) = \sum_{s=0}^{p-1} h_s(z)\cdot w^{(n-p+1)^2+(n-s)\cdot(m-n)}\cdot{n-s-1 \choose p-s-1}(w^2).\]
\end{theorem}

To prove the theorem, note that since the system of ideals $\{I_{(d^p)}:d\geq 0\}$ is cofinal with the one consisting of powers of the ideal of $p\times p$ minors, we obtain from \cite[Ex.~A1D.1]{eis-syzygies} that for each $i=0,1,\cdots,m\cdot n,$	 
\[H^i_{I_p}(S)=\varinjlim_{d}\Ext^i_S(S/I_{(d^p)},S),\]
where the successive maps in the directed system are induced from the natural quotient maps $S/I_{((d+1)^p)}\onto S/I_{(d^p)}$. By Theorem~\ref{thm:injectivityExt}, all these maps are injective, so the description of the character of $H^i_{I_p}(S)$ can be deduced from Theorem~\ref{thm:ExtIx}. Note that the partitions $\x$ to which we apply Theorem~\ref{thm:ExtIx} have the property that $x_1=\cdots=x_p=d$, and $x_i=0$ for $i>p$. Since we are interested in the limit as $d\to\infty$, we might as well assume that $x_1=\cdots=x_p=\infty$, in which case (\ref{eq:restrllIx:a}) becomes vacuous. In what follows, $\ll$ will always be assumed to be a dominant weight.

If $s\leq t_j$ then $s+1\leq t_j+j$ for every $j=1,\cdots,n-p,$ so we get
\[\ll_{t_j+j}\overset{(\ll\in\bb{Z}^n_{dom})}{\leq} \ll_{s+1}\overset{(\ref{eq:restrllIx:c})}{\leq} s-m\overset{(s\leq t_j)}{\leq} t_j-m,\]
i.e (\ref{eq:restrllIx:c}) implies (\ref{eq:restrllIx:b}) (note that $x_{n+1-j}=0$ for $j\leq n-p$). We conclude that
\[
\begin{split}
H_p(z,w)&=\sum_{\substack{0\leq s\leq t_1\leq\cdots\leq t_{n-p}\leq p-1 \\ \ll_s\geq s-n \\ \ll_{s+1}\leq s-m}}[S_{\ll(s)}F\oo S_{\ll}G]\cdot z^{|\ll|}\cdot w^{m\cdot n+1-p^2-s\cdot(m-n)-2\cdot\left(\sum_{j=1}^{n-p}t_j\right)} \\
\overset{(\ref{eq:defhs})}{=} & \sum_{0\leq s\leq t_1\leq\cdots\leq t_{n-p}\leq p-1}h_s(z)\cdot w^{m\cdot n+1-p^2-s\cdot(m-n)-2\cdot\left(\sum_{j=1}^{n-p}t_j\right)}\\
\overset{(t'_j:=p-1-t_j)}{=}& \sum_{s=0}^{p-1}h_s(z)\cdot w^{(n-p+1)^2+(n-s)\cdot(m-n)}\cdot\sum_{p-1-s\geq t'_1\geq\cdots\geq t'_{n-p}\geq 0}w^{2\cdot\left(\sum_{j=1}^{n-p}t'_j\right)} \\
\overset{(\ref{eq:gauss})}{=} & \sum_{s=0}^{p-1} h_s(z)\cdot w^{(n-p+1)^2+(n-s)\cdot(m-n)}\cdot {n-s-1 \choose p-s-1}(w^2).
\end{split}
\]

\section*{Acknowledgments} 
This work was initiated while we were visiting the Mathematical Sciences Research Institute, for whose hospitality we are grateful. Special thanks go to Emily Witt who participated in the initial stages of this project. We would also like to thank David Eisenbud, Steven Sam, Anurag Singh and Uli Walther for helpful conversations. Experiments with the computer algebra software Macaulay2 \cite{M2} have provided numerous valuable insights. The first author acknowledges the support of the National Science Foundation Grant No.~1303042. The second author acknowledges the support of the Alexander von Humboldt Foundation, and of the National Science Foundation Grant No.~0901185.


\end{document}